\documentclass[english,12pt]{amsart}

\usepackage{amssymb,latexsym}
\usepackage{enumerate}
\usepackage[english]{babel}

\makeatletter
\@namedef{subjclassname@2010}{%
  \textup{2010} Mathematics Subject Classification}
\makeatother

\DeclareMathOperator{\dist}{\dist}

\DeclareMathOperator{\res}{\upharpoonright}

\newtheorem{theorem}{Theorem}
\newtheorem{lemma}[theorem]{Lemma}

\newtheorem{corollary}[theorem]{Corollary}

\title{Norming subspaces of Banach spaces}
\dedicatory{Dedicated to the memory of Professor Joseph Diestel}

\author{V. P. Fonf}
\address{Department of Mathematics, Ben-Gurion University of the Negev, 84105 Beer-Sheva, Israel}
\email{fonf@math.bgu.ac.il}
\author{S. Lajara}
\address{Departamento de Matem\'aticas\\ Universidad de Castilla la Mancha\\ Escuela de Ingenieros Industriales, 02071 Albacete, Spain}
\email{sebastian.lajara@uclm.es}
\author{S. Troyanski}
\address{Departamento de Matem\'aticas, Universidad de Murcia, Campus de Espinardo, 30100 Murcia, Spain,
and Institute of Mathematics and Informatics, Bulgarian Academy of Science, bl. 8, acad. G. Bonchev str., 1113 Sofia, Bulgaria}
\email{stroya@um.es}
\author{C. Zanco}
\address{Dipartimento di Matematica, Universit\`a degli Studi, Via C. Saldini, 50, 20133 Milano MI, Italy}
\email{clemente.zanco@unimi.it}

\thanks{Acknowlegdements. V. P. Fonf and C. Zanco were partially supported
by the Gruppo Nazionale per l'Analisi Matematica, la Probabilit\`a e le loro Applicazioni (GNAMPA) of the Istituto Nazionale di Alta Matematica
(INdAM) of Italy. S. Lajara was supported by
MICINN projects MTM2014-54182-P and MTM2015-65825-P (Spain), by MTM2017-86182-P (AEI/FEDER, UE) and by the Fundaci\'on S\'eneca (Agencia de Ciencia y Tecnolog\'ia de la Regi\'on de Murcia) under project 19275/PI/14.  S. Troyanski was supported by
MICINN project MTM2014-54182-P (Spain), by MTM2017-86182-P (AEI/FEDER, UE),  by the Fundaci\'on S\'eneca (Agencia de Ciencia y Tecnolog\'ia de la Regi\'on de Murcia) under project 19275/PI/14 and by the Bulgarian National Scientific Fund. DFNI-I02/10, 2015.}

\date{\today}

\begin{document}

\baselineskip=15pt
\begin{abstract}
We show that, if $X$ is a closed subspace of a Banach space $E$
and $Z$ is a closed subspace of $E^*$ such that $Z$ is norming for
$X$ and $X$ is total over $Z$ (as well as $X$ is norming for $Z$
and $Z$ is total over $X$), then $X$ and the pre-annihilator of
$Z$ are complemented in $E$ whenever $Z$ is $w^*$-closed or $X$ is
reflexive.
\end{abstract}
\maketitle


Let $E$ be a Banach space, let $X$ be a subspace of $E$ and $Z$ be
a subspace of $E^*$ (the dual space to $E$). We say that \emph{$Z$ is norming for $X$} if the
formula
$$|||x||| = \sup_{f\in B_Z} |f(x)|, \,\,\, x\in X$$
defines an equivalent norm on $X$ (where $B_Z$ denotes the unit
ball of $Z$). It is clear that if $Z$ is norming for $X$, then
$Z$ is total over $X$ (that is, $X\cap Z_{\perp} = \{0\}$, where
$Z_{\perp} = \{x \in E: f(x) = 0 \ {\rm for \ every} \ f \in Z\})$. Analogously, if \emph{$X$ is norming for $Z$} (namely, if
the image of $X$ through the canonical maping $\pi:E\to E^{**}$ is norming for $Z$), then $X$ is total over $Z$
(that is, $X^{\perp}\cap Z = \{0\}$, where $X^{\perp} = \{f \in
E^*: f(x) = 0 \ {\rm for \ every} \ x \in X\}$).


A systematic treatment of these properties was carried out in the paper
\cite{DDL}, devoted to the study of norming bibasic systems in
Banach spaces. In particular, there it was shown that, if $X$ is a
closed subspace of of Banach space $E$ and $Z$ is a closed
subspace of $E^*$, then $X$ is norming for $Z$ if and only if the
restriction to $Z$ of the restriction map $q^*:E^*\to X^*$ is an
isomorphic embedding, if and only if $X$ is total over $Z$ and the
(direct) sum $X^{\perp} \oplus Z$ is closed in $E^*$.


In this note, we obtain a generalization of this statement, which
provides a characterization of the property that $Z$ is norming
for $X$ whenever the ball $B_Z$ is $w^*$-dense in
$B_{\overline{Z}^{w^*}}$. This result will be used to show that,
in the case $Z$ is $w^*$-closed, the pair of conditions ``$X$ is
total over $Z$" and ``$Z$ is norming for $X$'' (as well as ``$Z$
is total over $X$" and ``$X$ is norming for $Z$") entail that $E =
X\oplus Z_{\perp}$. Using a duality argument, we deduce that the
same assertion holds true in the case $X$ is reflexive. As an
application of this result, we provide a criterion for the
existence of a sequence of extensions $\{f_i\}_{i=1}^{\infty}$ of
the functionals associated to an $M$-basis from a reflexive
subspace of $X$ of a separable Banach space $E$ with the property
that $[f_i]$ (the closed linear span of $\{f_i\}_i$) is norming
for $E$.



\

The main result of this note reads as follows.

\begin{theorem}\label{main1}
Let $E$ be a Banach space and $X$ be a closed subspace of $E$. If
$Z$ is a $w^*$-closed subspace of $E^*$, then the following
assertions are equivalent:
\begin{enumerate}
\item $Z$ is norming for $X$ and $X$ is total over $Z$.
\item $X$ is norming for $Z$ and $Z$ is total over $X$.
\item $X$ is norming for $Z$ and $Z$ is norming for $X$.
\item $E=X\oplus Z_{\perp}$.
\end{enumerate}
\end{theorem}
In the proof of this result we shall use the following lemma.

\begin{lemma}\label{t1}
Let $E$ be a Banach space. If $X$ is a closed subspace of $E$ and $Z$ is a closed subspace of $E^*$, then the following conditions are equivalent:
\begin{enumerate}
\item[(a)] $\overline{Z}^{w^*}$ is norming for $X$.
\item[(b)] The restriction to $X$ of the quotient map $Q:E\to E/Z_{\perp}$ is an isomorphic embedding.
\item[(c)] $Z$ is total over $X$ and the (direct) sum $X\oplus Z_{\perp}$ is closed in $E$.
\end{enumerate}
If in addition, $B_Z$ is $w^*$-dense in $B_{\overline{Z}^{w^*}}$, then these conditions are equivalent to:
\begin{enumerate}
\item[(d)] $Z$ is norming for $X$.
\end{enumerate}
\end{lemma}
\begin{proof} $(a)\Rightarrow (b)$ Let us write $F= \overline{Z}^{w^*}$, and let $\lambda\in (0,1]$ be a number
satisfying $$\sup_{f\in B_F} |f(x)|\geq \lambda \|x\| \quad
\text{for every} \quad  x\in X.$$ Fix $x\in X$ and pick $f\in B_F$
such that $f(x) \geq \lambda\|x\|/2$. As $(Z_{\perp})^{\perp} =
F$, it easily follows that $F_{\perp}= Z_{\perp}$. Therefore, for
each $y\in Z_{\perp}$ we have $$\|x-y\|\geq f(x-y) = f(x)\geq
\lambda \|x\|/2$$ and consequently, $\|Qx\| = \inf \left\{\|x-y\|:
\, y\in Z_{\perp}\right\} \geq \lambda \|x\|/2.$ Hence the
operator $Q_{\res X}$ is an isomorphic embedding.

\

$(b)\Rightarrow (c)$ It suffices to show that $\inf \{\|x-y\|:\,
x\in S_X, \, y\in S_{Z_{\perp}}\}> 0,$ where $S_X$ and
$S_{Z_{\perp}}$ denote respectively the unit spheres of $X$ and
$Z_{\perp}$. Assume the contrary. Then there exist sequences
$(x_n)_n\subset S_X$ and $(y_n)_n\subset S_{Z_{\perp}}$ such that
$\|x_n-y_n\|\to 0$. Thus, $\|Qx_n-Qy_n\|\to 0$, and hence $\|Qx_n\|\to
0$.  But, because of our assumption we have $\|Qx_n\|\geq \lambda \|x_n\|=
\lambda$ for some $\lambda\in (0,1]$ and every $n\in \mathbb N$.
Therefore, the manifold $X\oplus Z_{\perp}$ is a closed subspace of
$E$.

\

$(c)\Rightarrow (a)$ Set $U= X\oplus Z_{\perp}$ and let $M$ and
$N$ denote the annihilator subspaces of $X$ and $Z_{\perp}$
relative to $U$, that is, $M = \left\{f\in U^*:\, f \res
X=0\right\}$ and $N = \left\{g\in U^*: \, g \res Z_{\perp}
=0\right\}$. Since $U$ is closed, according to \cite[Excercise
4.36]{Fabian} it follows that $U^* = M\oplus N$. In particular,
there exists $\alpha> 0$ such that
$$\alpha \left(\|f\|+\|g\|\right)\leq \|f+g\|\leq \|f\|+ \|g\| \quad {\rm whenever}
 \quad f\in M\,\, \text{and}\,\, g\in N.$$
Choose a vector $x\in X$ with $\|x\|=1$, take $\varphi\in U^*$
with $\varphi(x)= \|\varphi\| = 1$ and let functionals $f\in M$
and $g\in N$ such that $\varphi = f+g$. It is clear that  $g(x) =
\varphi(x) = 1$ and, because of the previous inequality, we have
$\|g\|\leq \alpha^{-1}$. Therefore, the functional $\psi = \alpha
g$ belongs to $B_{N}$ and $\psi(x)\geq \alpha$. Now, let
$\widehat{\psi}\in E^*$ be such that $\widehat{\psi} \res U =
\psi$ and $\|\widehat{\psi}\| = \|\psi\|$. Then, $\widehat{\psi}
\in B_{(Z_{\perp})^{\perp}} = B_{\overline{Z}^{w^*}}$ and
$\widehat{\psi}(x)\geq \alpha$. Consequently, the subspace $
\overline{Z}^{w^*}$ is norming for $X$.

\

Finally, it is clear that $(d)$ implies $(a)$ (with no additional
assumption). Further, assuming $\overline{B_{Z}}^{w^*} =
B_{\overline{Z}^{w^*}}$, we have ${\small \sup \limits_{f\in B_Z}
f(x) = \sup \limits_{f\in B_{\overline{Z}^{w^*}}}  f(x)}$ for
every $x\in X$, hence $(a)\Rightarrow (d)$. \end{proof}

\noindent \emph{Remarks.} (1) In general, without the assumption
$\overline{B_Z}^{w*} = B_{\overline{Z}^{w^*}}$,  assertion $(d)$
in the previous lemma is not implied by the other ones. Indeed, if
$E$ is any non quasi-reflexive Banach space then, according to the
main result in \cite{DL} there exists a closed subspace $Z\subset
E^*$ such that $Z$ is total but not norming for $E$. Hence
$\overline{Z}^{w^*} = E^*$, so $\overline{Z}^{w^*}$ is norming for
$E$.

\

(2) If  $E$ is a weakly compactly generated Banach space not
isomorphic to a Hilbert space, then for every non-complemented
subspace $X\subset E$ there is a $w^*$-closed subspace $Z\subset
E^*$ such that $Z$ is total but not norming for $X$. Indeed,
thanks to \cite[Theorem 13.48]{Fabian} there exists a subspace
$Y\subset E$ such that $X\cap Y = \{0\}$ and $X+ Y$ is dense in
$E$. Set $Z = Y^{\perp}$. Then $Z$ is a $w^*$-closed subspace of
$E^*$ and $X\cap Z_{\perp} = X\cap Y = \{0\}$. Hence $Z$ is total
over $X$. Since $X+ Z_{\perp}$ is a proper dense manifold in $E$,
it follows that it is not closed. By using Lemma \ref{t1} we
deduce that $Z$ is not norming for $X$.

\

As a particular case of the former lemma we get the aforementioned result from \cite{DDL}.
\begin{corollary}{\rm (\cite[Theorem 2]{DDL})}\label{ddl-th2} Let $E$ be a Banach space. If $X$ is a closed subspace of $E$ and $Z$ is a closed subspace of $E^*$, then the following conditions are equivalent:
\begin{enumerate}
\item[(1)] $X$ is norming for $Z$.
\item[(2)] The restriction to $Z$ of the restriction map $q^*:E^*\to X^*$ is an isomorphic embedding.
\item[(3)] $X$ is total over $Z$ and the (direct) sum $X^{\perp}\oplus Z$ is closed in $E^*$.
\end{enumerate}
\end{corollary}
\begin{proof}
Put $\widetilde{E} = E^*$, $\widetilde{X} = Z$ and $\widetilde{Z}
= \pi(X)$. Then, $\widetilde{Z}_{\perp}=X^{\perp}$. Hence,
assertion $(1)$ is satisfied if and only if $\widetilde{Z}$ is
norming for $\widetilde{X}$, and condition $(3)$ is equivalent to
the properties that $\widetilde{Z}$ is total over $\widetilde{X}$
and $\widetilde{Z}_{\perp}\oplus \widetilde{X}$ is closed in
$\widetilde{E}$. Further, since $E^*/X^{\perp}\cong X^*$, the
quotient map $Q:\widetilde{X}\to
\widetilde{E}/\widetilde{Z}_{\perp}$ can be identified with the
restriction operator $q^*_{\res Z}:Z\to X^*$, thus condition $(2)$
is equivalent to the fact that $Q:\widetilde{X}\to
\widetilde{E}/\widetilde{Z}_{\perp}$ is an isomorphic embedding.
Finally, thanks to Goldstine's theorem we have
$\overline{B_{\widetilde{Z}}}^{w^*} =
B_{\overline{\widetilde{Z}}^{w^*}}$ and Lemma \ref{t1} applies.
\end{proof}

\noindent \emph{Proof of Theorem \ref{main1}.} It is clear that
$(4)$ entails $(1)$ and $(2)$. Thus, it is enough to prove the
implications $(1)\Rightarrow (4)$, $(2)\Rightarrow (4)$ and
$(4)\Rightarrow (3)$.

\

$(1) \Rightarrow (4)$  Because of the hypothesis we have $X\cap
Z_{\perp} = \{0\}$. We claim that the direct sum $X\oplus
Z_{\perp}$ is dense in $E$. Indeed, since $Z$ is $w^*$-closed, the
adjoint operator of the map $Q_{\res X}: X\to E/Z_{\perp}$ can be
identified with the restriction map $q^*_{\res Z}: Z\to X^*$. It
is clear that $\ker q^*_{\res Z} = X^{\perp}\cap Z$. Bearing in
mind that $X$ is total over $Z$, it follows that $q^*_{\res Z}$ is
one-to-one. Hence, the operator $Q_{\res X}$ has dense range, and
using the Hahn-Banach theorem we deduce that the manifold $X\oplus
Z_{\perp}$ is dense in $E$. On the other hand, as $Z$ is norming
for $X$, Lemma \ref{t1} guarantees that $X\oplus Z_{\perp}$ is
closed. Consequently, $E= X\oplus Z_{\perp}$.

\

$(2)\Rightarrow (4)$ As $Z$ is $w^*$-closed we have $Z =
(Z_{\perp})^{\perp}$. Therefore, according to \cite[Excercise
4.16]{Fabian}, it is enough to show that $E^* = X^{\perp}\oplus Z$,
and this is clearly equivalent to the fact that the operator
$q^*_{\res Z}$ is an isomorphism from $Z$ onto $X^*$. Since $X$ is
norming for $Z$, Corollary \ref{ddl-th2} yields the existence of a
number $\lambda> 0$ such that $$\|q^*(z)\|\geq \lambda \|z\| \quad
\text{ for every } \quad z\in Z.$$ Therefore, by the open mapping
theorem, it is enough to check that $q^*_{\res Z}$ is onto. We
claim that $M = q^*(Z)$ is a $w^*$-closed subspace of $X^*$.
According to the Banach-Dieudonn\'{e} Theorem (see e.g.
\cite[Theorem 3.92]{Fabian}), it is enough to prove that the set
$B_M$ is $w^*$-closed. Let $\{x_{\alpha}^*\}_{\alpha\in \Lambda}$
be a net in $B_M$ such that
$x_{\alpha}^*\stackrel{\text{w}^*}{\longrightarrow} x^*$ for some
$x^*\in X^*$. For each $\alpha\in \Lambda$ there is (a unique)
$z_{\alpha}\in Z$ with $q^*(z_{\alpha}) = x_{\alpha}^*$. By the
previous inequality we have $\|z_{\alpha}\|\leq \lambda^{-1}$ for
each $\alpha\in \Lambda.$ Thus the net $\{z_{\alpha}\}_{\alpha}$
has a $w^*$-cluster point, say $z\in E^*$. As $Z$ is $w^*$-closed,
we get $z\in Z$. Moreover, since the map $q^*$ is $w^*$-$w^*$
continuous, we have
$q^*(z_{\alpha})\stackrel{\text{w}^*}{\longrightarrow} q^*(z)$,
that is $x_{\alpha}^*\stackrel{\text{w}^*}{\longrightarrow}
q^*(z)$. Consequently, $x^*= q^*(z)$, so $x^*\in B_M$. Therefore,
$B_M$ (hence also $M$) is $w^*$-closed. Since $M$ is also is total
over $X$ we have $M = X^*$.

\

$(4)\Rightarrow (3)$ Taking into account that $Z$ is $w^*$-closed
and the manifold $X\oplus Z_{\perp}$ is closed in $E$, from Lemma
\ref{t1} we deduce that $Z$ is norming for $X$. Moreover, a new
appeal to \cite[Exercise 4.16]{Fabian} yields $E^* =
X^{\perp}\oplus Z$, which implies that $X$ is norming for $Z$.
\hfill $\square$

\

The next result constitutes an analogue of Theorem \ref{main1} in
case $X$ is reflexive.

\begin{corollary}\label{main2}
Let $E$ be a Banach space, let $X$ be a closed subspace
of $E$ and $Z$ be a closed subspace of $E^*$. If $X$ is reflexive then the following
conditions are equivalent:
\begin{enumerate}
\item $X$ is norming for $Z$ and $Z$ is total over $X$.
\item $Z$ is norming for $X$ and $X$ is total over $Z$.
\item $X$ is norming for $Z$ and $Z$ is norming for $X$.
\item $Z$ is $w^*$-closed and $E=X\oplus Z_{\perp}$.
\end{enumerate}
\end{corollary}
\begin{proof}
Notice that the facts that $X$ is reflexive and norming for $Z$ entail that $Z$
is $w^*$-closed. Indeed, in such a case, by Corollary
\ref{ddl-th2} the map $q^*_{\res Z}: Z\to X^*$ is an isomorphic
embedding. Therefore, $q^*(Z)$ is a closed subspace of $X^*$ that
is isomorphic to $Z$. Since $X^*$ is reflexive, $q^*(Z)$ is
reflexive as well, and using \cite[Lemma 4.62]{Fabian} we deduce
that $Z$ is $w^*$-closed. Thus Theorem \ref{main1} yields that
$(1)$ implies $(4)$. The reverse implication is obvious. Hence, it
remains to show that assertions $(1)$, $(2)$ and $(3)$ are
equivalent. This follows by duality applying Theorem \ref{main1}
to space $\widetilde{E} = E^*$ and to subspaces $\widetilde{X} =
Z\subset \widetilde{E}$ and $\widetilde{Z} = \pi(X) \subset
\widetilde{E\,}^*$. Observe that, since $X$ is reflexive,
$\widetilde{Z}$ is a $w^*$-closed subspace of $\widetilde{E\,}^*$.
\end{proof}

\noindent \emph{Remarks.} (1) A Banach space $E$ is called
\emph{indecomposable} if $E$ cannot be written as the direct sum
of two infinite dimensional closed subspaces. The first example of
an indecomposable Banach space was constructed by Gowers and
Maurey (c.f. \cite{GM}): it enjoys the much stronger property of
being \emph{hereditarily indecomposable} (i.e., every infinite
dimensional closed subspace of that space is indecomposable). As
an immediate consequence of Lemma \ref{t1}, the fact (pointed out
by V. D. Milman, c.f. \cite[Lemma 1.1]{AT}) follows that a Banach
space $E$ is hereditarily indecomposable if, and only if, for any
closed subspace $X\subset E$ with $\dim (X) = \infty$ and each
$w^*$-closed subspace $Z\subset E^*$ such that $Z$ is norming for
$X$, we have ${\rm codim}\, (Z) < \infty$. Analogously, an easy
consequence of Theorem \ref{main1} yields the following
characterization of indecomposable spaces: A Banach space $E$ is
indecomposable if, and only if, for every closed subspace $X
\subset E$ with $\dim (X) = \infty$ and every $w^*$-closed
subspace $Z\subset E^*$ such that $Z$ is norming for $X$ and $X^{\perp}\cap Z = \{0\}$ (as well as $X$ is norming for $Z$ and $X \cap Z_{\perp}=\{0\}$) we
have ${\rm codim}\, (Z) < \infty$.

\

(2) A Banach space $E$ is said to have the \emph{Dunford-Pettis property} (in short, DPP),
if for any two weakly null sequences $\{x_n\}_n\subset E$ and $\{x_n^*\}_n\subset E^*$
we have $x_n^*(x_n)\to 0$ (for some equivalent formulations of this property see
e.g. \cite[Theorem 1]{Di-survey}). Typical examples of spaces with this
property are $L^1(\Omega,\mu)$, where $(\Omega,\mu)$ is any $\sigma$-finite measure space,
and $\mathcal{C}(K)$ for any compact Hausdorff space $K$. If $E$ is a Banach space
with the DPP and $X$ is a reflexive infinite dimensional subspace of $E$, then:
\begin{enumerate}
\item[$(a)$] for every closed subspace $Z\subset E^*$ such that $Z$ is
norming for $X$, we have that $X$ is not total over $Z$, and
\item[$(b)$] for every closed subspace $Z\subset E^*$ such that $X$ is
norming for $Z$, we have that $Z$ is not total over $X$.
\end{enumerate}
Indeed, being reflexive and infinite-dimensional, $X$ fails to
have the DPP (c.f. \cite[p. 597]{Fabian}). According to a clasical
result by Grothendieck (c.f. \cite[Lemma 13.44]{Fabian}) it
follows that $X$ is not complemented in $E$ and Corollary
\ref{main2} applies.

\

We end this note with an application of Corollary \ref{main2} in
the setting of $M$-bibasic systems in separable Banach spaces.
Recall that a sequence $\{x_i\}_{i=1}^{\infty}$ in a Banach space
$E$ is a \emph{Markushevich basis} (in short, \emph{$M$-basis}) of
$E$ provided that $E = [x_i]$ and there exists a (unique) sequence
of functionals $\{x_i^*\}_i\subset E^*$ such that $\{x_i,
x_i^*\}_i$ is a biorthogonal system in $E$ and the subspace
$[x_i^*]\subset E^*$ is total over $E$ (we refer to \cite[Section
1.f]{LTbook} for the fundamental properties of $M$-bases in Banach
spaces). We say that a biorthogonal system $\{x_i,
z_i\}_{i=1}^{\infty}$ in $E$ is \emph{$M$-bibasic} whenever
$\{x_i\}_i$ is an $M$-basis of $[x_i]$ and $\{z_i\}_i$ is an
$M$-basis of $[z_i]$. If $\{x_i\}_i$ and $\{z_i\}_i$ are basic
sequences, the system $\{x_i, z_i\}_i$ is called \emph{bibasic}.
It was shown in \cite{DDL} that every infinite-dimensional Banach
space has a bibasic system $\{x_i, z_i\}_{i=1}^{\infty}$ such that
$\sup_{i} \|x_i\|\|z_i\|< \infty$ ($\{x_i, z_i\}_i$ is said to be
\emph{bounded}) and $[z_i]$ is not norming for $[x_i]$. Actually,
the existence of norming bibasic (or $M$-bibasic systems) is a
rather strong condition (for instance, from the previous remark it
follows that, if $E$ is a Banach space with the DPP and $X$ is a
reflexive subspace of $E$, then no $M$-bibasic system $\{x_i,
z_i\}_i$ exists in $E$ with $\{x_i\}_i \subset X$ and such that
$[z_i]$ is norming for $[x_i]$). In fact, in some cases the
presence of an $M$-bibasic system $\{x_i, z_i\}_i$ with this
property yields the existence of a biorthogonal sequence of
extensions of the functionals $x_i^*$ which is norming for the
whole space.

\begin{corollary}
Let $X$ be a reflexive subspace of a separable Banach space $E$
and  $\{x_i, z_i\}_{i=1}^{\infty}\subset X\times E^*$ be a
bounded $M$-bibasic system. If $[z_i]$ is norming for $[x_i]$ then there
exists a bounded sequence $\{f_i\}_i\subset E^*$ such that $\{x_i, f_i\}_i$ is
a biorthogonal system and
$[f_i]$ is norming for $E$.
\end{corollary}
\begin{proof}
We can assume that $X = [x_i]$. The separability of $E$ yields the
existence of a normalized sequence $\{u_j\}_j$ in $X^{\perp}$
which is $w^*$-dense in $B_{X^{\perp}}$ and such that, for every
$j$, the vector $u_j$ appears infinitely many times in that
sequence. Let us write, for each $i\in \mathbb N$, $f_i = z_i +
u_i.$ It is clear that $\{x_i, f_i\}_i$ is a biorthogonal system
in $E$. Put $Z= [z_i]$ and let $N$ denote the $w^*$-sequential
closure of ${\rm span}\, \{f_i\}_i$. Fix $j\in \mathbb N$ and let
$\{i_k\}_k$ be a strictly increasing sequence of positive integers
such that $u_{i_k} = u_j$ for every $k\in \mathbb N$. As $E$ is
separable and, by Corollary \ref{main2}, $Z$ is $w^*$-closed, we
have that $B_Z$ is $w^*$-sequentially compact. Thus, we can assume
that $z_{i_k}\stackrel{\text{w}^*}{\longrightarrow} z$ for some
$z\in Z$. Since $\{x_i\}_i$ is an $M$-basis of $X$ and $X$ is
total over $Z$, we easily get $z=0$, hence
$z_{i_k}\stackrel{\text{w}^*}{\longrightarrow} 0$. Consequently,
$f_{i_k}\stackrel{\text{w}^*}{\longrightarrow} u_j$. In particular
$u_j\in N$, therefore $z_j = f_j -u_j\in N$ for every $j\in
\mathbb N$. So $Z\subset N$. Further, bearing in mind that the
sequence $\{u_j\}_j$ is $w^*$-dense in $B_{X^{\perp}}$, we have
$X^{\perp}\subset N$. Since, because of Corollary \ref{main2},
$E^* = X^{\perp}\oplus (Z_{\perp})^{\perp} = X^{\perp}\oplus Z$,
it follows that $E^* = N$. Therefore, ${\rm span}\, \{f_i\}_i$ is
$w^*$-sequentially dense in $E^*$. Taking into account that $E$ is
separable, according to a result by Banach (c.f. \cite[Annexe,
Th\'eor\`eme 2]{B}), we deduce that $[f_i]$ is norming for $E$.
\end{proof}

\end{document}